\documentclass{amsart}

\usepackage{purrsonal}
\usepackage{tikz}
\usepackage{titlesec}

\titleformat{name=\section}{\centering\scshape}{\thetitle.}{0.5em}{}
\titleformat{name=\subsection}[runin]{\itshape}{\thetitle.}{0.5em}{}[.]
\titleformat{name=\subsubsection}[runin]{}{\thetitle.}{0.5em}{\itshape}[.]

\theoremstyle{plain}
\newcounter{probcounter}
\newenvironment{prob}[1][\unskip]{\refstepcounter{probcounter}\par\medskip
   \noindent \textbf{Problem~\theprobcounter.} \ifstrequal{#1}{\unskip}{#1}{(#1)} \rmfamily\itshape}{\medskip\par}

\title{On roots of Wiener polynomials of trees}
\author{Danielle Wang}

\begin{document}
\begin{abstract}
The \emph{Wiener polynomial} of a connected graph $G$ is
the polynomial $W(G;x) = \sum_{i=1}^{D(G)} d_i(G)x^i$
where $D(G)$ is the diameter of $G$, and $d_i(G)$ is the 
number of pairs of vertices at distance $i$ from each other.
We examine the roots of Wiener polynomials of trees.
We prove that the collection of real Wiener roots of trees
is dense in $(-\infty, 0]$, and the collection of complex
Wiener roots of trees is dense in $\CC$. We also prove that the
maximum modulus among all Wiener roots of trees of order
$n \ge 31$ is between $2n-15$ and $2n-16$, and we determine the unique
tree that achieves the maximum for $n \ge 31$. Finally,
we find trees of arbitrarily large diameter whose
Wiener roots are all real.
\end{abstract}

\maketitle

\section{Introduction}
\label{intro}
\begin{definition}
The \emph{Wiener polynomial} of a connected graph $G$ is 
\[
	W(G;x) = \sum_{i=1}^{D(G)} d_i(G) x^i
\]
where $D(G)$ denotes the diameter of $G$, and $d_i(G)$
is the number of unordered pairs of $G$
with distance $i$.
The \emph{reduced} Wiener polynomial of $G$ is
$$\widehat W(G;x) = W(G;x)/x.$$
We are concerned with the roots of this polynomial, which
we call \emph{Wiener roots}, especially for trees.
\end{definition}

Wiener polynomials were introduced in \cite{hosoya1988some}
and \cite{sagan1996wiener}.
Wiener polynomials are related to a quantity called
the \emph{Wiener index} of a connected graph, which originated
in chemical graph theory and is defined
to be the sum of the distances between all pairs of vertices
\cite{wiener1947structural}. The Wiener index was originally defined
for molecular graphs and is of interest because
it is closely correlated with the boiling point and several 
other physical properties of the substance \cite{gutmana1995wiener,rouvray1976dependence,stiel1962normal,wiener1947structural}.

It is easy to see that the Wiener index
of a graph is equal to the derivative of its Wiener polynomial
evaluated at $x = 1$.
In this paper, we exclusively study Wiener polynomials
of trees, so we remark that Wiener polynomials of trees also
arise naturally in network reliability, in the following way.
Given a graph where each edge
is operational with probability $p$, the \emph{resilience}
of $G$ is the expected number of pairs of vertices that
can communicate \cite{amin1993expected,colbourn1987network}. For a tree $T$, notice that
the probability that two vertices $u$ and $v$ can communicate
is exactly $p^{d(u,v)}$, where $d(u,v)$ is the
distance between $u$ and $v$. Thus, the resilience of $T$
is precisely $W(T;p)$.

Several properties of graph polynomials have been successfully
studied for chromatic \cite{fengming2005chromatic}, characteristic \cite{mowshowitz1972characteristic},
independence \cite{levit2005independence}, and reliability \cite{brownnetwork} polynomials of graphs.
Some interesting properties include the size of roots, 
the closure of the collection of roots, and the realness of 
the roots of these polynomials.
The roots of Wiener polynomials were studied in \cite{brown2018roots}. 
Wiener polynomials have also been studied in 
\cite{dehmer2012location,guo2006hyper,walikar2006hosoya,yan2007behavior}.

In 2018, Brown, Mol, and Oellermann
bounded the modulus of Wiener roots of connected graphs and determine
the unique graph with the maximum modulus Wiener root
\cite[Theorem 2.2]{brown2018roots}.
They also showed that the closure of the collection real Wiener
roots of connected graphs is the interval $(-\infty, 0]$
\cite[Theorem 3.1]{brown2018roots},
and the closure of the collection of real Wiener roots of 
trees contains the interval $(-\infty, -1] $
\cite[Theorem 3.3]{brown2018roots}.
They also prove that the collection of Wiener roots is not 
contained in any half-plane of $\CC$ \cite[Theorem 4.3]{brown2018roots}. 

They authors of \cite{brown2018roots} propose the following problems.

\begin{restatable}[{\cite[Problem 5.2]{brown2018roots}}]{probbb}{probreal}
\label{prob:real}
	Is the closure of real Wiener roots of trees the entire
	interval $(-\infty, 0]$?.
\end{restatable}

\begin{restatable}[{\cite[Problem 5.3]{brown2018roots}}]{probbb}{probcomplex}
\label{prob:complex}
	Is the closure of the collection of Wiener roots of
	connected graphs (or trees) the entire complex plane?
\end{restatable}

\begin{restatable}[{\cite[Problem 5.1]{brown2018roots}}]{probbb}{probmaxmod}
\label{prob:maxmod}
	What is the tree of order $n$ with the Wiener root of
	largest modulus?
\end{restatable}

\begin{restatable}[{\cite[Problem 5.5]{brown2018roots}}]{probbb}{probrealroots}
\label{prob:realroots}
	Which graphs have the property that their Wiener roots are
	all real? Which trees have this property?
\end{restatable}

In this paper we answer Problems \ref{prob:real}, \ref{prob:complex},
and \ref{prob:maxmod} from
\cite{brown2018roots} and provide a construction related to
Problem \ref{prob:realroots}.

\subsection{Outline}
Our main results are Theorems \ref{thm:real}, \ref{thm:complex},
and \ref{thm:maxmod}. 
In Section \ref{real}, we prove
Theorem \ref{thm:real} which solves Problem \ref{prob:real}.

\begin{restatable}{thmm}{real}
\label{thm:real}
The closure of the collection of real Wiener roots of 
trees is $(-\infty, 0]$.
\end{restatable}

The next theorem, which we prove in Section \ref{complex}, answers Problem \ref{prob:complex}.

\begin{restatable}{thmm}{complex}
\label{thm:complex}
The collection of Wiener roots of trees is dense in $\CC$.
\end{restatable}

In Section \ref{maxmod}, we prove Theorem \ref{thm:maxmod}
which solves Problem \ref{prob:maxmod}. 
The tree $T_n$ is defined in
Figure \ref{fig:Tnclone}, and we also show that the maximum modulus
of its roots is between $2n-15$ and $2n-16$ for $n$ large.

\begin{figure}[h]
\centering
\begin{tikzpicture}
	\fill (-3,0) circle(2.5pt);
	\fill (-2,0) circle(2.5pt);
	\fill (-1,0) circle(2.5pt);
	\fill (0,0) circle(2.5pt);
	\fill (1,0) circle(2.5pt);
	\fill (2,0) circle(2.5pt);
	\fill (3,0) circle(2.5pt);
	\fill (0,0.8) circle(2.5pt);
	\fill (-1,1.7) circle(2.5pt);
	\fill (-0.3,1.7) circle(2.5pt);
	\node at (0.3,1.7) {$\dots$};
	\node at (0, 1.95) {$\overbrace{\hspace{2.1cm}}$};
	\node at (0, 2.3) {$n - 8$};
	\fill (1,1.7) circle(2.5pt);

	\draw (-3,0) -- (-2,0) -- (-1,0) -- (0,0) -- (1,0) -- (2,0) -- (3,0);
	\draw (0,0.8) -- (0,0);
	\draw (-1,1.7) -- (0,0.8);
	\draw (-0.3, 1.7) -- (0, 0.8);
	\draw (1, 1.7) -- (0, 0.8);
\end{tikzpicture}
\caption{The tree $T_n$ with maximum modulus Wiener root.}
\label{fig:Tnclone}
\end{figure}
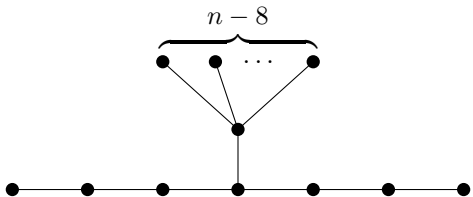

\begin{restatable}{thmm}{maxmod}
\label{thm:maxmod}
For $n \ge 31$,  
the unique tree of order $n$ with a Wiener root of maximum modulus
is $T_n$.
\end{restatable}

The authors of \cite{brown2018roots}
mention a construction for trees of arbitrarily
large diameter with all real Wiener roots. This
construction is incorrect. However, in Section \ref{realroots}
we show that there are
indeed trees of any diameter with all real Wiener roots.

\begin{restatable}{propp}{proprealroots}
\label{prop:realroots}
For any $D \ge 3$, there exists a tree with diameter $D$
whose Wiener polynomial has all real roots.
\end{restatable}

Finally, in Section \ref{open}, we present a number of open problems.

\section{Density of real Wiener roots}
\label{real}
In this section we prove Theorem \ref{thm:real}.

\real*
\begin{proof}
	By \cite[Theorem 3.3]{brown2018roots}, the closure
	contains $(-\infty, -1]$, so we just need to show that
	it also contains $[-1,0]$.

	Consider the tree $T_{k,n}$ given by Figure \ref{fig:Tkn}.
	\begin{figure}[h]
	\centering
	\begin{tikzpicture}
		\fill (-2,0) circle(2.5pt);
		\fill (-1, 0) circle(2.5pt);
		\node at (-0.5,0) {$\dots$};
		\node at (-0.52, -0.28) {$\underbrace{\hspace{2.8cm}}$};
		\node at (-0.5, -0.6) {$k$ edges};
		\fill (0,0) circle(2.5pt);
		\fill (1, 0) circle(2.5pt);
		\fill (2, 0.8) circle(2.5pt);
		\fill (2, 0.25) circle(2.5pt);
		\node at (2, -0.18) {$\vdots$};
		\fill (2, -0.8) circle(2.5pt);
		\draw (-2,0) -- (-1, 0);
		\draw (0,0) -- (1,0);
		\draw (1,0) -- (2, 0.8);
		\draw (1,0) -- (2, 0.25);
		\draw (1, 0) -- (2, -0.8);
		\node at (2, 0) {$\left. \begin{array}{c} 
		\phantom{x} \\ \phantom{x} \\ \phantom{x} \\
		\phantom{x}
		\end{array}\right\}$};
		\node at (2.6, 0) {$n$};
	\end{tikzpicture}
	\caption{The tree $T_{k,n}$.}
	\label{fig:Tkn}
	\end{figure}
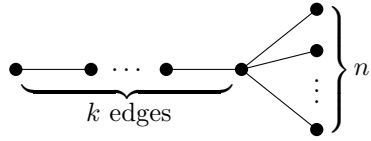
	Let $W_{k,n}$ denote the reduced Wiener polynomial of 
	$T_{k,n}$.

	We have for $k \ge 1$
	\begin{align*}
		W(k,n,x) &\coloneqq W_{k,n}(x) = (k+n) + \left( \tfrac{n^2+n}{2} + k -1 \right)
		x + (n+k-2)x^2 + \cdots + nx^k \\
		&= (\tfrac x2) n^2 + (1 + \tfrac{x}{2} + x^2 + \cdots
		+ x^k)n + (k + (k-1)x + \cdots + x^{k-1}).
	\end{align*}
	Observe that this is a polynomial in $n$ and can hence be
	defined for all real numbers $n$.
	Let
	\[
		n(x,k) = \frac{-R 
		- \sqrt{R^2 - 2x S}}{x},
	\]
	where $R = 1 + \tfrac x2 + x^2 + \cdots + x^k$ and
	$S = k + (k-1)x + \cdots + x^{k-1}$, so that
	$W(k,n(x,k),x) = 0$.

	We are concerned with values of $x$ in $(-1,0)$, in
	which case
	\[
		0 < R = \frac{x^{k+1}-1}{x-1} - \frac x2 <
		\frac{2}{1-x} - \frac{x}{2} \eqqcolon C_0(x)
	\]
	and
	\[
		S = \frac{k}{1-x} + \frac{x^{k+1}-x}{(1-x)^2} > 0.
	\]
	In particular, $n(x,k) \in \RR_{>0}$ for all
	$k \in \ZZ_{>0}$, $x \in (-1,0)$. 

	Let $n'(x,k) = \sqrt k f(x)$ where $f(x) = -\frac1x
	\sqrt{\frac{2x}{x-1}}$. We will show that for any
	$\delta > 0$, there exists $N_{x,\delta}$ such that
	\[
		k > N_{x,\delta} \implies |n(x,k) - n'(x,k)|
		< \delta - \frac{C_0(x)}{x}.
	\]
	We have
	\begin{align*}
		| n(x,k) - n'(x,k) |  &\le
		-\frac{R}{x} - \frac{\left| -\sqrt{R^2 - 2xS} + \sqrt{\frac{2kx}{x-1}}\right|}{x} \\
		&\le -\frac{C_0(x)}{x}
		-\frac1x \left(
		\sqrt{\frac{2kx}{x-1} + \frac{2x^2(1 - x^k)}{(1-x)^2} + R^2} -  \sqrt{\frac{2kx}{x-1}}\right) \\
		&\le -\frac{C_0(x)}{x} -\frac1x \frac{(C_0(x)^2 + 4x^2)\sqrt{x-1}}{2\sqrt{2kx}}.
	\end{align*}
	(The last inequality comes from the fact that
	$\sqrt{a + b} - \sqrt a < b/(2\sqrt a)$.)
	It is clear that if $k$ is large enough, the second term
	is less than $\delta$. Fix $\delta = 1$.

	Let $\epsilon > 0$ with $x + \epsilon < 0$
	We necessarily have $f(x) \neq f(x + \epsilon)$, so
	for $k > N_{x, \delta}, N_{x+\epsilon,\delta}$ large enough, we have
	\begin{align*}
		|n'(x,k) - n'(x+\epsilon, k)| &= \sqrt k(f(x) - f(x+\epsilon)) \\
		&> 1 + \left(\delta - \frac{C_0(x)}{x}\right)
	+ \left(\delta - \frac{C_0(x+\epsilon)}{x + \epsilon}\right).
	\end{align*}
	This implies that
	$|n(x,k) - n(x+\epsilon, k)| > 1$. Since $n(x,k)$ is
	a continuous function in $x \in (-1, 0)$ for fixed $k$,
	there exists $0 < \epsilon_0 < \epsilon$ such that
	$n(x+\epsilon_0, k)$ is a positive integer.
	Then $x + \epsilon_0$ is a Wiener root of
	$T_{k,n(x+\epsilon_0, k)}$. Since this holds for
	arbitrarily small $\epsilon$, $x$ is in the closure of
	the collection of real Wiener roots of trees.
\end{proof}

\section{Density of complex Wiener roots}
\label{complex}
In this section we prove Theorem \ref{thm:complex}.

\complex*
\begin{proof}
	Let $k, n \in \ZZ_{>0}$, and let $a$ be a sequence
	$a_1, \dots, a_n$ of positive integers. Let $T = T_{a,k,n}$
	be the graph shown in Figure \ref{fig:Gakn}.
	For a positive integer $c$, let $ca$ denote the
	sequence $ca_1,\dots, ca_n$. Let $W_{a,k,n}$ denote
	the reduced Wiener polynomial of $T_{a,k,n}$.

	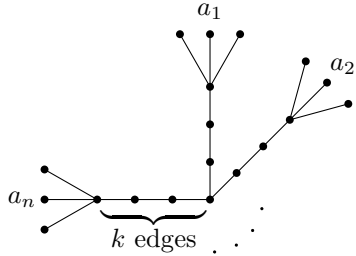
\begin{figure}[h]
	\centering
	\begin{tikzpicture}
		\fill (0,0) circle(1.5pt);
		\fill (0,0.5) circle(1.5pt);
		\fill (0,1) circle(1.5pt);
		\fill (0,1.5) circle(1.5pt);
		\fill (0,2.2) circle(1.5pt);
		\fill (0.4,2.2) circle(1.5pt);
		\fill (-0.4,2.2) circle(1.5pt);
		\draw (0,0) -- (0,0.5) -- (0,1) -- (0, 1.5) -- (0, 2.2);
		\draw (-0.4, 2.2) -- (0,1.5) -- (0.4, 2.2);
		\fill[rotate=90] (0,0.5) circle(1.5pt);
		\fill[rotate=90] (0,1) circle(1.5pt);
		\fill[rotate=90] (0,1.5) circle(1.5pt);
		\fill[rotate=90] (0,2.2) circle(1.5pt);
		\fill[rotate=90] (0.4,2.2) circle(1.5pt);
		\fill[rotate=90] (-0.4,2.2) circle(1.5pt);
		\draw[rotate=90] (0,0) -- (0,0.5) -- (0,1) -- (0, 1.5) -- (0, 2.2);
		\draw[rotate=90] (-0.4, 2.2) -- (0,1.5) -- (0.4, 2.2);
		\fill[rotate=-45] (0,0.5) circle(1.5pt);
		\fill[rotate=-45] (0,1) circle(1.5pt);
		\fill[rotate=-45] (0,1.5) circle(1.5pt);
		\fill[rotate=-45] (0,2.2) circle(1.5pt);
		\fill[rotate=-45] (0.4,2.2) circle(1.5pt);
		\fill[rotate=-45] (-0.4,2.2) circle(1.5pt);
		\draw[rotate=-45] (0,0) -- (0,0.5) -- (0,1) -- (0, 1.5) -- (0, 2.2);
		\draw[rotate=-45] (-0.4, 2.2) -- (0,1.5) -- (0.4, 2.2);
		\node at (0,2.5) {$a_1$};
		\node at (45:2.5) {$a_2$};
		\node at (180:2.5) {$a_n$};

		\fill[rotate=-10] (0.7,0) circle(0.6pt);
		\fill[rotate=-36] (0.7,0) circle(0.6pt);
		\fill[rotate=-60] (0.7,0) circle(0.6pt);
		\fill[rotate=-85] (0.7,0) circle(0.6pt);

		\node at (-0.75,-0.22) {$\underbrace{\hspace{1.4cm}}$};
		\node at (-0.75, -0.55) {$k$ edges};
	\end{tikzpicture}
	\caption{The tree $T_{a,k,n}$.}
	\label{fig:Gakn}
	\end{figure}

	We have $D(T) = 2k+2$,
	\[ d_{2k+2}(T) =
	\sum_{i \neq j} a_ia_j, \;\; d_2(T) = \sum_{i=1}^n
	\frac{a_i^2}{2} + \ell_2(a,k,n),
	\] 
	and for 
	$i \neq 2, 2k+2$ we have $d_i(T) = \ell_i(a,k,n)$, where the
	functions $\ell_i(a,k,n)$ are linear in the $a_i$.

	Let $M_{a,k,n} = \frac{W_{a,k,n}}{d_{2k+2}(T)}$
	and $R(a) = \frac{\sum_{a_i^2}}{2\sum_{i\neq j}a_ia_j}$.
	Then
	\[
		M_{a,k,n}(x) = x^{2k+1}
		+ \frac{\ell_{2k}(a,k,n)}{\sum a_ia_j} x^{2k}
		+ \cdots 
		+ \left(R(a) + \frac{\ell_2(a,k,n)}{\sum a_ia_j} \right)x
		+ \frac{\ell_1(a,k,n)}{\sum a_ia_j}.
	\]
	Since the $\ell_i$ are linear, we see that the
	sequence of polynomials $M_{ca, k, n}$ for
	$c = 1, 2, \dots$ approaches the polynomial $x^{2k+1}
	+ R(a) x$ under the norm
	$\| c_mx^m + \cdots + c_0 \| = \max_{0 \le i \le m} |c_i|$.

	For fixed $n$, it is straightforward to show that
	the values $R(a)$ can take are dense in 
	$[\frac{1}{n-1}, \infty)$ (the minimum is achieved when
	all the $a_i$ are equal), so the values $R(a)$ can
	take are dense in $[0, \infty)$ as $n$ and $a$ vary. 
	It is also straightforward to show that the roots of 
	$x^{2k} + r$ for $k \in \ZZ_{>0}, r \in \RR_{> 0}$ are
	dense in $\CC$. Thus, by continuity of roots
	(see for example \cite{harris1987shorter}) the roots of
	$M_{a,k,n}(x)$, which are the Wiener roots of $T_{a,k,n}$
	are dense in $\CC$ as $a$, $k$, $n$ vary.
\end{proof}

\section{Maximum modulus Wiener root of trees}
\label{maxmod}
Let $T_n$ be the tree shown in Figure \ref{fig:Tn},
and let $W_n$ be the reduced Wiener polynomial of $T_n$.
In this section, we show that for $n \ge 31$,
the unique tree of
order $n$ with the Wiener root of largest modulus is $T_n$.

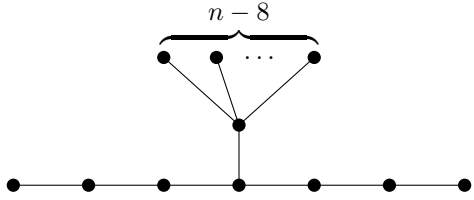
\begin{figure}[h]
\centering
\begin{tikzpicture}
	\fill (-3,0) circle(2.5pt);
	\fill (-2,0) circle(2.5pt);
	\fill (-1,0) circle(2.5pt);
	\fill (0,0) circle(2.5pt);
	\fill (1,0) circle(2.5pt);
	\fill (2,0) circle(2.5pt);
	\fill (3,0) circle(2.5pt);
	\fill (0,0.8) circle(2.5pt);
	\fill (-1,1.7) circle(2.5pt);
	\fill (-0.3,1.7) circle(2.5pt);
	\node at (0.3,1.7) {$\dots$};
	\node at (0, 1.95) {$\overbrace{\hspace{2.1cm}}$};
	\node at (0, 2.3) {$n - 8$};
	\fill (1,1.7) circle(2.5pt);

	\draw (-3,0) -- (-2,0) -- (-1,0) -- (0,0) -- (1,0) -- (2,0) -- (3,0);
	\draw (0,0.8) -- (0,0);
	\draw (-1,1.7) -- (0,0.8);
	\draw (-0.3, 1.7) -- (0, 0.8);
	\draw (1, 1.7) -- (0, 0.8);
\end{tikzpicture}
\caption{The tree $T_n$ with maximum modulus Wiener root.}
\label{fig:Tn}
\end{figure}

\begin{lemma}
\label{lem:bigroot}
	For $n \ge 10$,  the tree $T_n$ has a real Wiener root $-(2n-15) < r < -(2n-16)$.
\end{lemma}
\begin{proof}
	We calculate
	\[
		W_n(x) = (n-1) + \left(\frac{n^2-15n+70}{2}\right)x
		+ (2n-10)x^2 + (2n-11)x^3 + (2n-14)x^4 + x^5.
	\]
	Thus,
	\[
		W_n(-(2n-16))
		= 16n^4 - 545n^3 + 6959n^2 - 39485n + 84015 > 0
	\]
	for $n \ge 10$. 
	Also,
	\[
		W_n(-(2n-15))
		= -25n^3 + \frac{1165 n^2}{2} - \frac{9063 n}{2}
		+ 11774 < 0
	\]
	for $n \ge 9$.
	This implies that $W_n$ has a root $-(2n-15) < r < -(2n-16)$.
\end{proof}

As in \cite{brown2018roots}, we will use the following theorem
to bound the roots of polynomials.

\begin{theorem}[Enestr\"om--Kakeya \cite{kakeya1912limits}]
	If $f(x) = a_0 + a_1x + \cdots + a_nx^n$ has positive 
	real coefficients, then all complex roots of $f$ lie
	in the annulus
	\[
		r \le |z| \le R
	\]
	where $r = \min\left\{ \frac{a_i}{a_{i+1}} \colon 0 \le i
	\le n-1 \right\}$ and $R = \max \left\{ \frac{a_i}{a_{i+1}}
	\colon 0 \le i \le n -1 \right\}$.
\end{theorem}

\begin{definition}
	We say that a tree is \emph{bad} if it is equal either to
	$T_n$ or one of the trees $T_n'$ or $T_n''$ in Figure \ref{fig:baddiam6} or 
	\ref{fig:baddiam4}, and it is \emph{good} otherwise.

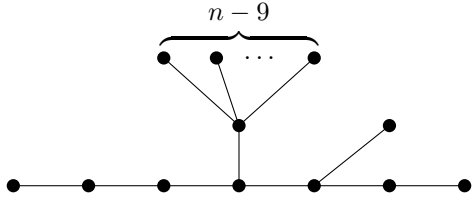
\begin{figure}[h]
\centering
\begin{tikzpicture}
	\fill (-3,0) circle(2.5pt);
	\fill (-2,0) circle(2.5pt);
	\fill (-1,0) circle(2.5pt);
	\fill (0,0) circle(2.5pt);
	\fill (1,0) circle(2.5pt);
	\fill (2,0) circle(2.5pt);
	\fill (3,0) circle(2.5pt);
	\fill (0,0.8) circle(2.5pt);
	\fill (-1,1.7) circle(2.5pt);
	\fill (-0.3,1.7) circle(2.5pt);
	\node at (0.3,1.7) {$\dots$};
	\node at (0, 1.95) {$\overbrace{\hspace{2.1cm}}$};
	\node at (0, 2.3) {$n - 9$};
	\fill (1,1.7) circle(2.5pt);

	\draw (-3,0) -- (-2,0) -- (-1,0) -- (0,0) -- (1,0) -- (2,0) -- (3,0);
	\draw (0,0.8) -- (0,0);
	\draw (-1,1.7) -- (0,0.8);
	\draw (-0.3, 1.7) -- (0, 0.8);
	\draw (1, 1.7) -- (0, 0.8);
	\fill (2,0.8) circle(2.5pt);
	\draw (1,0) -- (2,0.8);
\end{tikzpicture}
\caption{Bad tree $T_n'$ with diameter $6$.}
\label{fig:baddiam6}
\end{figure}

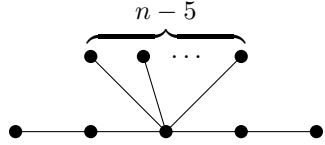
\begin{figure}[h]
\centering
\begin{tikzpicture}
	\fill (-2,0) circle(2.5pt);
	\fill (-1,0) circle(2.5pt);
	\fill (0,0) circle(2.5pt);
	\fill (1,0) circle(2.5pt);
	\fill (2,0) circle(2.5pt);
	\fill (-1,1) circle(2.5pt);
	\fill (-0.3,1) circle(2.5pt);
	\node at (0.3,1) {$\dots$};
	\node at (0, 1.25) {$\overbrace{\hspace{2.1cm}}$};
	\node at (0, 1.6) {$n - 5$};
	\fill (1,1) circle(2.5pt);

	\draw (-2,0) -- (-1,0) -- (0,0) -- (1,0) -- (2,0);
	\draw (-1,1) -- (0,0);
	\draw (-0.3, 1) -- (0, 0);
	\draw (1, 1) -- (0, 0);
\end{tikzpicture}
\caption{Bad tree $T_n''$ with diameter $4$.}
\label{fig:baddiam4}
\end{figure}
\end{definition}

Lemma \ref{lem:induct} follows from the proof of \cite[Lemma 2.3]{brown2018roots}. We will modify this
proof to show Lemma \ref{lem:2n-16}.

\begin{lemma}
\label{lem:induct}
	Let $T$ be a tree with diameter $D$, let $v$
	be a leaf on a diametral path of $T$, and let $T' = T - v$.
	Let $2 \le k < D(T')$. Then
	\[
		\frac{d_k(T)}{d_{k+1}(T)} \le 
		\max\left\{\frac{d_k(T')}{d_{k+1}(T')}, n - D\right\}.
	\]
	Furthermore, if $T$ has a unique diametral path between leaves
	$u$ and $v$, then every path of length $D - 1$
	contains exactly one of $u$ and $v$.
\end{lemma}
\begin{proof}
	Let $u$ be the neighbor of $v$ in $T$, and let
	$d_\ell(T',u)$ denote the number of paths of length
	$\ell \ge 1$ in $T'$ that have $u$ has an end-vertex.
	Note that $u$ is on a path of length $D-1$ in $T'$, so
	there are at most $n-1 - (D-1) = n - D$ vertices
	at distance $\ell$ from $u$ in $T'$ for
	$1 \le \ell \le D-1$. Thus we have
	$1 \le d_\ell(T',u) \le n - D$ for $1\le \ell \le D-1$.
	Let $s = \max\left\{ \frac{d_k(T')}{d_{k+1}(T')}, n-D\right\}$.
	Then for $2 \le k < D(T')$,
	\begin{align*}
		d_k(T) &= d_k(T') + d_{k-1}(T',u) \\
		&\le sd_{k+1}(T') + d_{k-1}(T',u) \\
		&= s[d_{k+1}(T) - d_k(T',u)] + d_{k-1}(T',u) \\
		&= sd_{k+1}(T) - sd_k(T',u) + d_{k-1}(T',u)\\
		&\le sd_{k+1}(T) - s + n-D \\
		&\le sd_{k+1}(T).
	\end{align*}
	For the second statement, suppose the path between
	two vertices $x$ and $y$ has length $D-1$.
	Let $P$ denote the path between $u$ and $v$.
	Let the shortest paths from $x$ and $y$ to $P$ meet $P$ at
	$w_x$ and $w_y$ respectively, and without loss of 
	generality, suppose that $u$, $w_x$, $w_y$, $v$ lie in that order
	on the path $P$. Then by the assumption
	that $P$ is the unique diametral path, we must 
	have $d(x,w_x) < d(u, w_x)$ and $d(y,w_y) < d(v,w_y)$.
	But then $d(x,y) \le d(u,v)-2 = D-2$, a contradiction.
\end{proof}

\begin{lemma}
\label{lem:2n-16}
	If $T$ is a good tree of order $n \ge 3$ and diameter $D$,
	then 
	\[ 	
		\frac{d_k(T)}{d_{k+1}(T)} \le \max\left\{\frac{3n-1}{2},
	2n-16\right\}
	\]
	for $1 \le k < D$. If $T$ is a bad tree with diameter $6$,
	then 
	\[
		\frac{d_k(T)}{d_{k+1}(T)} \le \max\left\{\frac{3n-1}{2},
	2n-14\right\}.	
	\]
\end{lemma}
\begin{proof}
	We induct on $n$. The case $n = 3$ is clear.
	Let $T$ be a tree of order $n > 3$ and assume that
	the lemma holds for all trees with $n-1$ vertices.
	Note that $\frac{d_1}{d_2} < 2 < \frac{3n-1}{2}$,
	so assume $k \ge 2$. Note that we will always have
	\[
		\max\left\{\frac{3(n-1)-1}{2}, 2(n-1)-16\right\}
		\ge n - D.
	\]

	If $T$ contains at least $2$ diametral paths, then
	there exists a leaf $v$ on a diametral path such that
	$T' = T - v$ has diameter $D$. If $T'$ is good or 
	bad with diameter $6$, then Lemma \ref{lem:induct}
	and the above inequality shows that
	\[
		\frac{d_k(T)}{d_{k+1}(T)} 
	\le \max \left\{ \frac{3n-4}{2}, 2(n-1)-14 \right\}
	\] 
	for $2 \le k < D$, so we are done. If $T'$ is bad with diameter $4$,
	then there are only two possibilities for $T$ (it is equal to
	$T_n''$ with a leaf attached either at one of the second-to-last 
	vertices on the	path of length $4$, or at one of the
	$n-5$ leaves), and we
	can check directly that they have Wiener polynomials
	\[
		(n-1)x + (\tfrac{n^2 - 9n + 28}{2})x^2
		+ (3n-15)x^3 + 2x^4
	\]
	and
	\[
		(n-1)x + (\tfrac{n^2-9n + 26}{2})x^2 + (3n-15)x^3
		+ 3x^4.
	\]
	It is easy to check that these coefficients
	satisfy $\frac{d_k(T)}{d_{k+1}(T)} \le \frac{3n-1}{2}$.

	Now suppose $T$ contains exactly $1$ diametral path,
	and let $u$	and $v$ be its endpoints.
	Then $T' = T - v$ has diameter $D - 1$.
	If $T'$ is bad with diameter $4$, then $T$ must be
	equal to $T_{n-1}''$ with a leaf attached to one of
	the endpoints of the path of length $4$, and it 
	has Wiener polynomial
	\[
		(n-1)x + (\tfrac{n^2-9n+26}{2})x^2 + (2n-9)x^3
		+ (n-4)x^4 + x^5.	
	\]
	It is straightforward to show that this polynomial satisfies
	the condition.

	If $T'$ is good or bad with diameter $6$, then by the
	same argument as above,
	\[
		\frac{d_k(T)}{d_{k+1}(T)} 
	\le \max \left\{ \frac{3n-4}{2}, 2(n-1)-14 \right\}
	\]
	for $2 \le k < D - 1$. 

	Thus, we only need to show that
	$d_{D-1}(T) \le \max \{ \frac{3n-1}{2}, 2n-16 \}$ if $T$
	is a good tree with a unique diametral path, and that
	$d_{D-1}(T)\le 2n-14$ if $T$ is a bad tree
	with diameter $6$. The second statement can be checked
	directly, since
	we have $d_{D-1}(T_n) = 2n-14$ and $d_{D-1}(T_n') = 2n-15$.
	For the first, note that if $T$ has diameter $2$ or $3$
	and has a unique diametral path, then it is a path and
	it is easy to check that $T$ satisfies the condition.
	If $T$ has diameter $4$ and a unique diametral path,
	then it is bad. Thus, we can assume that $D(T) > 4$. 

	Let $a$ be the number of vertices of $T$ (other than
	$u$ and $v$) which are
	endpoints of two paths of length $D-1$, and let $b$
	be the number of vertices which are endpoints of $1$ path 
	of length $D-1$.
	We have $a + b \le n - D + 1$, since $D-1$ of
	the vertices on the path from $u$ to $v$ are not counted at all, and $b \ge 2$, since two vertices on
	the path from $u$ to $v$ are endpoints of exactly
	$1$ path of length $D-1$.
	Thus, $d_{D-1}(T) = 2a + b \le 2(n-D)$, so we can
	assume $5 \le D \le 7$.

	If $D$ is odd, we must have $a = 0$ since every path
	of length $D-1$ must contain exactly one of $u$ and $v$,
	yet for any vertex $x$, $d(u,x) \neq d(v,x)$ by parity
	considerations.
	Thus, $2a + b \le n - D + 1 \le \frac{3n-1}{2}$.

	Thus, the only remaining case is $D = 6$. If $a = 0$ we
	are done, so assume $a \ge 1$. If a vertex $x$
	has distance $5$ from both $u$ and $v$, then $x$ must be
	a leaf whose neighbor $y$ has distance $4$
	from both $u$ and $v$. Then $a + b \le n - 6$
	so $2a + b \le 2(n-8) + 2 = 2n-14$. The only way we
	can have $2a + b = 2n-14$ is if $a = n-8$. This means
	that every other vertex must be a leaf adjacent to $y$,
	but then $T = T_n$ is bad. The only way we can have
	$2a + b = 2n-15$ is if $a = n-9$, $b = 3$. But then
	$T = T_n'$, so it is bad. Thus, $2a + b \le 2n-16$.
\end{proof}

\begin{lemma}
\label{lem:baddiam4}
For $n \ge 10$,
the tree $T_n''$ does not have a Wiener root with modulus 
greater than $\left(1 + \frac{1}{\sqrt2}\right)n - 7$.
\end{lemma}
\begin{proof}
	The reduced Wiener polynomial of $T_n''$ is
	$(n-1) + (\frac{n^2 - 7n + 16}{2})x + (2n-8)x^2 + x^3$.
	The discriminant of this cubic is greater than $0$ for $n \ge 10$.
	Thus, it has $3$ real, negative roots which
	sum to $-(2n-8)$. By \cite[Thoerem 2.6]{brown2018roots},
	it has one root $-\left(1 + \frac{1}{\sqrt2}\right) n + 7 < r < -\left(1 + \frac{1}{\sqrt2}\right)n + 8$.
	Therefore $r$ and the other roots each have modulus at most
	$2n-8 - \left(1 + \frac{1}{\sqrt2}\right)n + 8
	< \left(1 + \frac{1}{\sqrt2}\right)n - 7$ for $n \ge 10$.
\end{proof}

\begin{lemma}
\label{lem:baddiam6}
	For $n \ge 31$, the tree $T_n'$ does not have a Wiener root with modulus
	greater than $2n-16$.
\end{lemma}
\begin{proof}
	Let $W(x)$ denote the reduced Wiener polynomial of $T_n'$.
	We have
	\[
		W(x) = (n-1) + \left(\tfrac{n^2-17n + 90}{2}\right)x
		+ (2n-9)x^2 + (3n-21)x^3 + (2n-15)x^4 + x^5.
	\]
	We find $W(-(2n-17)) > 0$ for $n \ge 13$ and
	$W(-(2n-16)) < 0$ for $n \ge 9$. Thus,
	$T_n'$ has a Wiener root $r$
	with $-(2n-16) < r < -(2n-17)$ for $n \ge 31$.

	Then $W(x)$ factors as
	\begin{align*}
		W(x) = (x-r)\Bigg( x^4 &+ (2n-15+r)x^3 
		+ (r(2n-15+r) + 3n-21)x^2 \\&+ \left(- \frac{n^2-17n+90}{2r} - \frac{n-1}{r^2} \right)x 
		+ \left(-\frac{n-1}{r} \right)\Bigg).
	\end{align*}
	Note that $|2n-15+r| < 2$, so
	$|r(2n-15+r) + 3n-21| < n-4$. Also, it is easy to show that
	$\left|- \frac{n^2-17n+90}{2r} - \frac{n-1}{r^2}\right|< \frac{n}{4}$ and $\left|-\frac{n-1}{r}\right| < 1$.
	Suppose that the second factor has a root $y$ with modulus greater
	than $2n-16 \ge 46$. Since for $n \ge 31$ we have 
	$n-4 < 2n-16 < |y| $ and $\frac{n}{4} < (2n-16)^2 < |y|^2$, we
	obtain
	\begin{align*}
		|y|^4 &\le 2|y|^3 + (n-4)|y|^2 + \frac{n}{4}|y| + 1 \\
		&\le \frac{2|y|^4}{46} + \frac{|y|^4}{46} +  \frac{|y|^4}{46}
		+ \frac{|y|^4}{46} \\
		&< |y|^4,
	\end{align*}
	a contradiction.
\end{proof}

\maxmod*
\begin{proof}
	For $n \ge 31$ we have $2n - 16 \ge \frac{3n-1}{2}$ and
	$2n-16 > \left(1 + \frac{1}{\sqrt2}\right)n - 7$. Thus,
	by Enestr\"om--Kakeya and Lemma \ref{lem:2n-16}, if $T$
	is good, then it does not have a Wiener root of modulus
	greater than $2n-16$. By Lemmas \ref{lem:baddiam4} and
	\ref{lem:baddiam6}, if $T$ is bad and not equal to $T_n$,
	then it does not have a Wiener root of modulus greater than
	$2n-16$ either.
	Thus by Lemma \ref{lem:bigroot}, $T_n$ is
	the only tree with a Wiener root of modulus greater
	than $2n-16$.
\end{proof}

\section{Trees with all real Wiener roots}
\label{realroots}
In this section we provide 
a construction for trees with arbitrarily large diameter and
all real Wiener roots. In \cite{brown2018roots}, Brown, Mol, and Oellermann give a construction which is incorrect (the claim that
$W(T_1;x) = (x+1)^2 W(T_0;x)$, where $T_1$ is the tree obtained
by adding a leaf to each of the vertices of $T_0$, is false).

\begin{theorem}[{\cite[Theorem 1]{kurtz1992sufficient}}]
\label{thm:allreal}
	Let $P$ be a polynomial of degree $n \ge 2$ with positive
	coefficients. If
	\[
		a_i^2 - 4a_{i-1}a_{i+1} > 0, \quad i = 1, 2,\dots, n-1,
	\]
	then all the roots of $P$ are real and distinct.
\end{theorem}

\proprealroots*
\begin{proof}
	Let $t > 4D^2$ be an integer, and for $1 \le i \le D-1$,
	let $a_i = t^{-\frac{i(i-1)}{2}}$, so $a_{i}^2 = ta_{i-1}a_{i+1}$
	for $2 \le i \le D-2$.
	Consider the tree $T$ in Figure \ref{fig:Treal}, where
	$n$ is sufficiently divisible by powers of $t$.
	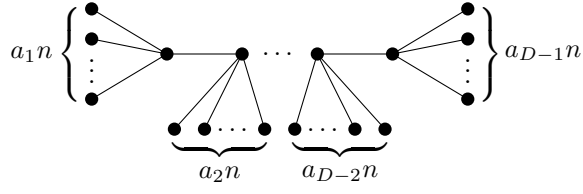
\begin{figure}[h]
	\centering
	\begin{tikzpicture}
		\fill (-2,0) circle(2.5pt);
		\fill (-1, 0) circle(2.5pt);
		\node at (-0.5,0) {$\dots$};
		\fill (0,0) circle(2.5pt);
		\fill (1, 0) circle(2.5pt);
		\fill (2, 0.6) circle(2.5pt);
		\fill (2, 0.2) circle(2.5pt);
		\node at (2, -0.12) {$\vdots$};
		\fill (2, -0.6) circle(2.5pt);
		\draw (-2,0) -- (-1, 0);
		\draw (0,0) -- (1,0);
		\draw (1,0) -- (2, 0.6);
		\draw (1,0) -- (2, 0.2);
		\draw (1, 0) -- (2, -0.6);
		\node at (2, 0) {$\left. \begin{array}{c} 
		\phantom{x} \\ \phantom{x} \\ \phantom{x}
		\end{array}\right\}$};
		\node at (3, 0) {$a_{D-1}n$};
		\fill (-3, 0.6) circle(2.5pt);
		\fill (-3, 0.2) circle(2.5pt);
		\node at (-3, -0.12) {$\vdots$};
		\fill (-3, -0.6) circle(2.5pt);
		\draw (-2,0) -- (-3, 0.6);
		\draw (-2,0) -- (-3, 0.2);
		\draw (-2, 0) -- (-3, -0.6);
		\node at (-3, 0) {$\left\{ \begin{array}{c} 
		\phantom{x} \\ \phantom{x} \\ \phantom{x}
		\end{array}\right.$};
		\node at (-3.8, 0) {$a_1n$};
		\fill (-1.9,-1) circle(2.5pt);
		\fill (-1.5, -1) circle(2.5pt);
		\node at (-1.1,-1) {$\dots$};
		\fill (-0.7, -1) circle(2.5pt);
		\draw (-1.5,-1)--(-1,0)--(-0.7,-1);
		\draw (-1.9,-1) -- (-1,0);
		\fill (-0.3,-1) circle(2.5pt);
		\node at (0.1,-1) {$\dots$};
		\fill (0.5,-1) circle(2.5pt);
		\fill (0.9, -1) circle(2.5pt);
		\draw (0.9,-1) -- (0,0);
		\draw (-0.3,-1)--(0,0)--(0.5,-1);

		\node at (-1.3,-1.25) {$\underbrace{\hspace{1.25cm}}$};
		\node at (-1.3, -1.6) {$a_2 n$};
		\node at (0.3,-1.25) {$\underbrace{\hspace{1.25cm}}$};
		\node at (0.3, -1.6) {$a_{D-2}n$};
	\end{tikzpicture}
	\caption{The tree $T$ with all real roots.}
	\label{fig:Treal}
	\end{figure}

	We have $d_1(T) = \ell_1(n)$,
	and $d_i(T) = c_in^2 + \ell_i(n)$, where 
	\[
		c_i = a_1a_{i-1} + a_2a_i + \cdots
	+ a_{D-i+1}a_{D-1}
	\] for $2 \le i \le D$, and
	the $\ell_i(n)$ are linear functions in $n$
	for $1 \le i \le D-1$.
	Note that 
	\[
	a_{i-1} = a_1a_{i-1} < c_i < D a_1a_{i-1} = D a_{i-1}.
	\]
	So 
	\[ 
		c_i^2 > a_{i-1}^2 = ta_{i-2}a_{i}
		> \frac{t}{D^2} c_{i-1}c_{i+1} > 4c_{i-1}c_{i+1},
	\]
	for $3 \le i \le D-2$. Thus, for $2 \le i \le D-2$,
	$d_i^2(T) - 4d_{i-1}(T)d_{i+1}(T)$ is a quartic polynomial
	in $n$ with positive leading coefficient, so
	for $n$ a sufficiently large multiple of $t^{\frac{(D-1)(D-2)}{2}}$, the coefficients of $\widehat W(T;x)$ satisfy the condition
	in Theorem \ref{thm:allreal}, and thus $T$ has all
	real Wiener roots.
\end{proof}

\section{Open problems}
\label{open}
We discuss some open problems and conjectures.
First, we have not resolved Problems 5.4 and 5.5 from
\cite{brown2018roots}.

\probrealroots*

\begin{prob}[{\cite[Problem 5.4]{brown2018roots}}]
	Which connected graph (or tree) of order $n$ has a Wiener root of 
	\textnormal{(i)}
	largest real part and \textnormal{(ii)} of largest imaginary part?
	What are the rates of growth, as a function of $n$, of these parameters?
\end{prob}

Note that our proof of Proposition \ref{prop:realroots} 
does not tell us anything about the roots other than that they
are real, and it also requires $n$ to be very large. 
We propose the following questions, which we think
may be interesting.

\begin{prob}
	For fixed $D$, what is the smallest connected graph (or tree) with diameter $D$	and all real Wiener roots?
\end{prob}

For $D = 2,3,4,5$, the smallest trees with diameter $D$ and
all real Wiener roots have sizes $3,7,10,15$ respectively.

\begin{prob}
	Do there exist connected graphs (or trees) with arbitrarily large diameter whose Wiener roots are all rational?
\end{prob}

\begin{prob}
	Which connected graphs (or trees) have Wiener
	polynomials with a double root? Are there trees with
	a repeated Wiener root other than $-1$?
\end{prob}
We note that there are $2$ trees of order $9$ and 
$54$ trees of order $16$ with
Wiener polynomials divisible by $(x+1)^2$. However, it
is not true for such trees $T$ that $W(T;x)/(x+1)^2$ is necessarily
the Wiener polynomial of any graph (it may have negative
coefficients). There are no trees of order $n \le 16$ with
a repeated Wiener root not equal to $-1$. We also do not
know of a general construction for trees with $-1$ as a
double root.

Calculations suggest the following conjecture, which may answer
part of \cite[Problem 5.4]{brown2018roots}.

\begin{conjecture}
	For $n \ge 9$, the tree with order $n$ with a Wiener root of largest
	imaginary part is the tree $\widetilde T_n$ shown in
	Figure \ref{fig:tilden}. It appears that this tree has a Wiener root whose real
	part approaches $-1/2$ and imaginary part which approaches
	$c\sqrt n$ asymptotically for some constant $c$. 
\begin{figure}[h]
\centering
\begin{tikzpicture}
	\fill (-3,0) circle(2.5pt);
	\fill (-2,0) circle(2.5pt);
	\fill (-1,0) circle(2.5pt);
	\fill (0,0) circle(2.5pt);
	\fill (1,0) circle(2.5pt);
	\fill (2,0) circle(2.5pt);
	\fill (3,0) circle(2.5pt);
	\fill (-1,1.2) circle(2.5pt);
	\fill (-0.3,1.2) circle(2.5pt);
	\node at (0.3,1.2) {$\dots$};
	\node at (0, 1.45) {$\overbrace{\hspace{2.1cm}}$};
	\node at (0, 1.8) {$n - 7$};
	\fill (1,1.2) circle(2.5pt);

	\draw (-3,0) -- (-2,0) -- (-1,0) -- (0,0) -- (1,0) -- (2,0) -- (3,0);
	\draw (-1,1.2) -- (0,0.0);
	\draw (-0.3, 1.2) -- (0, 0.0);
	\draw (1, 1.2) -- (0, 0.0);
\end{tikzpicture}
\caption{The tree $\widetilde T_n$ with maximum imaginary
part Wiener root, for $9 \le n \le 18$.}
\label{fig:tilden}
\end{figure}
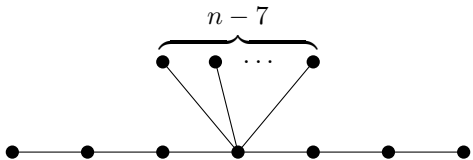
\end{conjecture}

The fact that $\widetilde T_n$ has the Wiener root with
maximum imaginary part has been verified for $9 \le n \le 18$.

\begin{remark}
	Finally, we have not found the tree of order $n$
	with maximum modulus Wiener root for $n \le 31$,
	but we suspect that it is always either $T_n$ or $T_n''$
	(Brown, Mol, and Oellermann verified that it is $T_n''$
	for $n \le 17$).
\end{remark}

\section{Acknowledgements}
This research was conducted at the University of Minnesota
Duluth REU and was supported by NSF / DMS grant 1650947 and 
NSA grant H98230-18-1-0010. I would like to thank
Joe Gallian for suggesting the problem, and I would like to
thank Aaron Berger, Joe Gallian, and Mitchell Lee for many
helpful comments on the paper.

\bibliography{wiener}{}
\bibliographystyle{abbrv}

\end{document}